\documentclass[11pt]{article}
\usepackage[utf8]{inputenc}
\usepackage{hyperref}
\usepackage{amsmath,amssymb,amsthm}
\usepackage{amsfonts,tikz}
\usepackage[margin=1in]{geometry}
\newcommand{\bk}{\backslash}
\newcommand{\lm}{\lambda}
\newcommand{\mt}{\widetilde{m}}

\newtheorem{lemma}[]{Lemma}
\newtheorem{cor}[lemma]{Corollary}
\newtheorem{theorem}[lemma]{Theorem}

\title{A complete multipartite basis for the chromatic symmetric function}
\author{Logan Crew, Sophie Spirkl\footnote{Department of Combinatorics \& Optimization, University of Waterloo, Waterloo, ON, N2L 3E9.\newline  Emails: lcrew@uwaterloo.ca,  sspirkl@uwaterloo.ca. \newline  This material is based upon work supported by the National Science Foundation under Award No. DMS-1802201. \newline We acknowledge the support of the Natural Sciences and Engineering Research Council of Canada (NSERC), [funding reference number RGPIN-2020-03912]. \newline
Cette recherche a été financée par le Conseil de recherches en sciences naturelles et en génie du Canada (CRSNG), [numéro de référence RGPIN-2020-03912].}}
\date{\today}

\begin{document}

\maketitle

\begin{abstract}

In the vector space of symmetric functions, the elements of the basis of elementary symmetric functions are (up to a factor) the chromatic symmetric functions of disjoint unions of cliques. We consider their graph complements, the functions $\{r_{\lm}: \lm \text{ an integer partition}\}$ defined as chromatic symmetric functions of complete multipartite graphs. This basis was first introduced by Penaguiao \cite{raul}. We provide a combinatorial interpretation for the coefficients of the change-of-basis formula between the $r_{\lm}$ and the monomial symmetric functions, and we show that the coefficients of the chromatic and Tutte symmetric functions of a graph $G$ when expanded in the $r$-basis enumerate certain intersections of partitions of $V(G)$.
    
\end{abstract}

\section{Introduction}

The theory of symmetric functions was reenergized in the 1980s by Macdonald, who discovered and collected information about bases and operators of various spaces of symmetric functions in his seminal works \cite{macpaper, mac}. This has led to a deeper understanding of symmetric group representations and algebraic geometry \cite{garsia, factorial}, and remains a very active area of research \cite{deltazero, delta, zabrocki}. Stanley connected this theory with colorings of a graph $G$ by introducing the chromatic symmetric function $X_G$ and the Tutte symmetric function $XB_G$ and showing that the coefficients of these functions in different symmetric function bases encode information about the underlying graph, well beyond the information encoded by the chromatic and Tutte polynomials \cite{stanley, stanley2}. Research on the chromatic symmetric function is also extremely active, with primary focus on the Stanley-Stembridge conjecture that unit interval graphs are $e$-positive \cite{huh, epos, dahl, foley}, the related conjecture that claw-free graphs are $s$-positive \cite{pau, paw, wang} and the Tree Isomorphism conjecture that nonisomorphic trees are distinguished by the chromatic symmetric function \cite{trees, heil}.  

Very recently, Penaguaio \cite{raul} demonstrated that the triangular modular relations for the chromatic symmetric function discovered by Orellana and Scott \cite{ore} (along with graph isomorphism) generate the kernel of the map induced by the chromatic symmetric function from the Hopf algebra of vertex-labelled graphs to symmetric functions. In other words, Penaguiao showed that in the Hopf algebra, a graph may be transformed into any other graph with the same chromatic symmetric function by applying a finite number of vertex relabellings and modular relations. The proof of this fact uses the triangular relation to reduce every chromatic symmetric function of a graph to a linear combination of those of complete multipartite graphs, suggesting that these particular functions may be interesting in this context.  Any complete multipartite graph is determined by the sizes of its disjoint maximal stable sets, so such graphs are naturally indexed by integer partitions.  Thus, Penaguiao introduced the family $r_{\lm}$ as the chromatic symmetric functions of these graphs\footnote{The authors are not aware of why Penaguiao chose the letter $r$. Some possiblilites we conjecture are that $r$ stands for \emph{rainbow}, \emph{Raul}, or \emph{really cool}.}, and wondered if they may carry further meaning in the context of all chromatic symmetric functions.

Simultaneously, the authors introduced a vertex-weighted version of the chromatic symmetric function \cite{delcon}, and showed that if for each integer partition $\lm = (\lm_1, \dots, \lm_k)$, the graph $K^{\lm}$ is the complete graph with $k$ vertices of weights $\lm_1, \dots, \lm_k$, then $X_{K^{\lm}}$ is equal to a multiple of the monomial symmetric function $m_{\lm}$. Moreover, the chromatic symmetric function of the graph complement $\overline{K^{\lm}}$ is equal to the power-sum symmetric function $p_{\lm}$. Given that the only other fundamental basis that may be written as (a multiple of) the chromatic symmetric function of a single graph is that of the elementary symmetric functions $e_{\lm}$ \cite{cho}, it naturally follows that the graph complements of these $e$-basis elements may be of interest. These complementary functions are none other than the $r_{\lm}$ described by Penaguiao.

In this paper, we show how this basis fits naturally with well-known symmetric function bases by showing that each $r_{\lm}$ expands into the $m$-basis with coefficients counting set partitions, and that each $m_{\lm}$ expands into the $r$-basis with coefficients enumerating necklace partitions, which add a cyclic ordering to the blocks of a set partition.  We also demonstrate a reciprocity theorem relating these coefficients to those arising from changing between the $p$- and $e$-bases. We show that the chromatic symmetric function of a graph $G$ expands into the $r$-basis with coefficients counting intersections of maximal stable partitions of $G$ by deriving a more general expression for the coefficients of the Tutte symmetric function expressed in the $r$-basis.

This paper will be structured as follows. In Section 2, we review necessary background on symmetric functions, graphs, and graph colorings, particularly the chromatic symmetric function. In Section 3, we define the symmetric functions $r_{\lm}$ and find combinatorial interpretations for the coefficients of the $r$-basis expansion of the chromatic symmetric function $X_G$ and the Tutte symmetric function $XB_G$. In Section 4, we give an interpretation of the entries of the transition matrices between the $r$-basis and the monomial basis, and show reciprocity relations with the transition matrices between the elementary and power sum bases.  Finally, in Section 5, we consider some further directions for research. 

\section{Background}

An \emph{integer partition} (or just \emph{partition}) is a tuple $\lm = (\lm_1,\dots,\lm_k)$ of positive integers such that $\lm_1 \geq \dots \geq \lm_k$.  The integers $\lm_i$ are the \emph{parts} of $\lm$.  If $\sum_{i=1}^k \lm_i = n$, we say that $\lm$ is a partition of $n$, and we write $\lm \vdash n$, or $|\lm| = n$.  The number of parts $k$ is the \emph{length} of $\lm$, and is denoted by $l(\lm)$.  The number of parts equal to $i$ in $\lm$ is given by $n_i(\lm)$.

A \emph{set partition} of a set $S$ is a set of nonempty \emph{blocks} $B_1,\dots,B_k$ such that each $B_i \subseteq S$, the blocks are pairwise disjoint, and their union is $S$. If $\pi$ is a partition of $S$, we write $\pi \vdash S$. The set $\{1,2,\dots,n\}$ will be denoted by $[n]$.

  A function $f(x_1,x_2,\dots) \in \mathbb{R}[[x_1,x_2,\dots]]$ is \emph{symmetric} if $f(x_1,x_2,\dots) = f(x_{\sigma(1)},x_{\sigma(2)},\dots)$ for every permutation $\sigma$ of the positive integers $\mathbb{N}$.  The \emph{algebra of symmetric functions} $\Lambda$ is the subalgebra of $\mathbb{R}[[x_1,x_2,\dots]]$ consisting of those symmetric functions $f$ that are of bounded degree (that is, there exists a positive integer $n$ such that every monomial of $f$ has degree $\leq n$).  Furthermore, $\Lambda$ is a graded algebra, with natural grading
  $$
  \Lambda = \bigoplus_{k=0}^{\infty} \Lambda^d
  $$
  where $\Lambda^d$ consists of symmetric functions that are homogeneous of degree $d$ \cite{mac,stanleybook}.

  Each $\Lambda^d$ is a finite-dimensional vector space over $\mathbb{R}$, with dimension equal to the number of partitions of $d$ (and thus, $\Lambda$ is an infinite-dimensional vector space over $\mathbb{R}$).  Some commonly-used bases of $\Lambda$ that are indexed by partitions $\lm = (\lm_1,\dots,\lm_k)$ include:
\begin{itemize}
  \item The monomial symmetric functions $m_{\lm}$, defined as the sum of all distinct monomials of the form $x_{i_1}^{\lm_1} \dots x_{i_k}^{\lm_k}$ with distinct indices $i_1, \dots, i_k$.

  \item The power-sum symmetric functions, defined by the equations
  $$
  p_n = \sum_{k=1}^{\infty} x_k^n, \hspace{0.3cm} p_{\lm} = p_{\lm_1}p_{\lm_2} \dots p_{\lm_k}.
  $$
  \item The elementary symmetric functions, defined by the equations
  $$
  e_n = \sum_{i_1 < \dots < i_n} x_{i_1} \dots x_{i_n}, \hspace{0.3cm} e_{\lm} = e_{\lm_1}e_{\lm_2} \dots e_{\lm_k}.
  $$
  \item The homogeneous symmetric functions, defined by the equations
  $$
  h_n = \sum_{i_1 \leq \dots \leq i_n} x_{i_1} \dots x_{i_n}, \hspace{0.3cm} h_{\lm} = h_{\lm_1}h_{\lm_2} \dots h_{\lm_k}.
  $$
\end{itemize}

  We also make use of the \emph{augmented monomial symmetric functions}, defined by 
  $$
  \mt_{\lm} = \left(\prod_{i=1}^{\infty} n_i(\lm)!\right)m_{\lm}.
  $$
 
 Given a symmetric function $g$ and a symmetric function basis $\{f_{\lm}\}$ indexed by partitions $\lm$, we use the notation $[f_{\mu}]g$ to mean the coefficient of $f_{\mu}$ when expressing $g$ in the $f$-basis.  This notation may also be used analogously for other functions (e.g. if $p(x)$ is a polynomial, we may use $[x^j]p(x)$ to denote its coefficient of $x^j$).

  A \emph{graph} $G = (V,E)$ consists of a \emph{vertex set} $V$ and an \emph{edge multiset} $E$ where the elements of $E$ are pairs of (not necessarily distinct) elements of $V$.  An edge $e \in E$ that contains the same vertex twice is called a \emph{loop}.  If there are two or more edges that each contain the same two vertices, they are called \emph{multi-edges}.  A \emph{simple graph} is a graph $G = (V,E)$ in which $E$ does not contain loops or multi-edges (thus, $E \subseteq \binom{V}{2}$).  If $\{v_1,v_2\}$ is an edge (or nonedge), we will write it as $v_1v_2 = v_2v_1$.  The vertices $v_1$ and $v_2$ are the \emph{endpoints} of the edge $v_1v_2$.  We will use $V(G)$ and $E(G)$ to denote the vertex set and edge multiset of a graph $G$, respectively.
  
  Two graphs $G$ and $H$ are said to be \emph{isomorphic} if there exists a bijective map $f: V(G) \rightarrow V(H)$ such that for all $v_1,v_2 \in V(G)$ (not necessarily distinct), the number of edges $v_1v_2$ in $E(G)$ is the same as the number of edges $f(v_1)f(v_2)$ in $E(H)$.

  The \emph{complement} of a simple graph $G = (V,E)$ is denoted $\overline{G}$, and is defined as $\overline{G} = (V, \binom{V}{2} \bk E)$, so in $\overline{G}$ every edge of $G$ is replaced by a nonedge, and every nonedge is replaced by an edge.

  A \emph{subgraph} of a graph $G$ is a graph $G' = (V',E')$ where $V' \subseteq V$ and $E' \subseteq E|_{V'}$, where $E|_{V'}$ is the set of edges with both endpoints in $V'$.  An \emph{induced subgraph} of $G$ is a graph $G' = (V',E|_{V'})$ with $V' \subseteq V$.  The induced subgraph of $G$ using vertex set $V'$ will be denoted $G|_{V'}$.  A \emph{stable set} of $G$ is a subset $V' \subseteq V$ such that $E|_{V'} = \emptyset$.  A \emph{clique} of $G$ is a subset $V' \subseteq V$ such that for every pair of distinct vertices $v_1$ and $v_2$ of $V'$, $v_1v_2 \in E(G)$.
  
  The \emph{complete graph} $K_n$ on $n$ vertices is the unique simple graph having all possible edges, that is, $E(K_n) = \binom{V}{2}$ where $V = V(K_n)$.

  Let $G = (V(G),E(G))$ be a (not necessarily simple) graph.  A map $\kappa: V(G) \rightarrow \mathbb{N}$ is called a \emph{coloring} of $G$.  This coloring is called \emph{proper} if $\kappa(v_1) \neq \kappa(v_2)$ for all $v_1,v_2$ such that there exists an edge $e = v_1v_2$ in $E(G)$.  The \emph{chromatic symmetric function} $X_G$ of $G$ is defined as
  
  \begin{equation}\label{eq:xgdef}
  X_G(x_1,x_2,\dots) = \sum_{\kappa} \prod_{v \in V(G)} x_{\kappa(v)}
  \end{equation} 
  where the sum runs over all proper colorings $\kappa$ of $G$.  Note that if $G$ contains a loop then $X_G = 0$, and $X_G$ is unchanged by replacing each multi-edge by a single edge.
  
  For a graph $G$ and an integer partition $\lm = (\lm_1,\dots,\lm_k)$, let $Stab_{\lm}(G)$ denote the set of (unordered) partitions of $V(G)$ into disjoint stable sets of sizes $\lm_1, \dots, \lm_k$. It is easy to see from \eqref{eq:xgdef} that we have the $\mt$-basis expansion \cite{stanley}
  \begin{equation}\label{eq:xgmt}
X_G = \sum_{\lm \vdash |V(G)|} |Stab_{\lm}(G)|\mt_{\lm}.
  \end{equation}

\begin{section}{Expanding the chromatic and Tutte symmetric functions using a multipartite basis}

For any integer partition $\lm = (\lm_1, \dots, \lm_k)$, define $G_{\lm}$ to be the simple graph with $V(G_{\lm}) = \{v_{11},\dots,v_{1 \lm_1},v_{21},\dots,v_{2\lm_2},v_{31},\dots,v_{k\lm_k}\}$ (thus $|V(G_{\lm})| = |\lm|)$, and $E(G_{\lm}) = \{v_{ij}v_{ab} \hspace{0.1cm} | \hspace{0.1cm} i \neq a\}$. Thus, $G_{\lm}$ is a \emph{complete multipartite graph}, meaning it consists of $k$ disjoint stable sets of sizes $\lm_1,\dots,\lm_k$, as well as all possible edges connecting vertices from different stable sets. Define the \emph{complete multipartite symmetric functions} by
$$
r_{\lambda} = X_{G_{\lm}}.
$$
Note that $X_{\overline{G_{\lm}}} = (\prod_{i=1}^{k} \lm_i!)e_{\lm}$, so the $r_{\lm}$ are in some sense the graph complement of the $e_{\lm}$.  Penaguiao showed that for each $d$, the set $\{r_{\lm}: \lm \vdash d\}$ is a basis\footnote{We provide a proof similar to Penaguiao's in Section 4.} for $\Lambda^d$ \cite{raul}.

Before exploring how the $r$-basis is related to other fundamental bases of $\Lambda$, we illustrate its usefulness by giving a combinatorial interpretation to the $r$-basis expansion of the both the chromatic and Tutte symmetric functions, answering a question of Penaguiao \cite{raul}.  

If $\pi$ is a set partition, let $\lm(\pi)$ denote the integer partition whose parts are the sizes of the blocks of $\pi$. The \emph{Tutte symmetric function} of a graph $G$ may be defined \cite{cstutte, jo} as
\begin{equation}\label{eq:xbtom}
XB_G(t;x_1,x_2,\dots) = \sum_{\pi \vdash V(G)} (1+t)^{e(\pi)}\mt_{\lm(\pi)}
\end{equation}
where $e(\pi)$ is the number of edges of $G$ that have both endpoints in the same block of $\pi$. Using the convention $0^0 = 1$, it is easy to see that setting $t = -1$ we get
$$
XB_G(-1;x_1,x_2,\dots) = X_G(x_1,x_2,\dots).
$$

Note that $XB_G$ is an element in the expanded ring $\Lambda[t]$, which may be considered either as polynomials in $t$ with symmetric function coefficients, or symmetric functions with coefficients in $\mathbb{R}[t]$. In particular, the set $\{1,(1+t),(1+t)^2,\dots\}$ forms a basis for $\mathbb{R}[t]$ as a vector space, and so given a basis $\{f_{\lm}\}$ for $\Lambda$, we may write each element of $\Lambda^d[t]$ uniquely as
$$\sum_{i \in \mathbb{N}} \sum_{\lm \vdash d} c_{\lm,i}(1+t)^if_{\lm}$$ for $c_{\lm,i} \in \mathbb{R}$. Since each element of $\Lambda[t]$ has bounded degree in $t$ and the $x_i$, the expressions $[f_{\lm}]XB_G$, $[(1+t)^i]XB_G$, and  $[(1+t)^if_{\lm}]XB_G$ are all well-defined.

The set $\Pi(G)$ of all partitions of $V(G)$ is a poset when equipped with the partial order relation $\pi_1 \leq \pi_2$ if $\pi_1$ is a refinement of $\pi_2$ (meaning each block of $\pi_1$ is entirely contained in a block of $\pi_2$).  Then $\Pi(G)$ has as the unique maximal element the partition $\hat{1}$ with one block containing all vertices, and as the unique minimal element $\hat{0}$ the partition with each vertex in its own block. For $\pi, \rho \in \Pi(G)$, let $\pi \wedge \rho$ denote the partition such that vertices $a$ and $b$ are in the same block if and only if they are in the same block in both $\pi$ and $\rho$. A partition $\pi \in \Pi(G)$ is called \emph{$j$-maximal} if $e(\pi) = j$, and there exists at least one edge between every pair of blocks of $\pi$. Thus, a $j$-maximal partition has $j$ edges inside of blocks and cannot be coarsened without adding additional such edges.

Before giving a combinatorial interpretation for the $r$-basis coefficients of $XB_G$, we prove an auxiliary lemma. 

\begin{lemma}\label{lem:lemcool}
Let $P$ be a partition of a set $S$ of size $n$.  Then for every $\mu \vdash n$, we have
\begin{equation}\label{eq:lemcool}
\sum_{\pi \leq P} [r_{\mu}]\mt_{\lm(\pi)} = \delta_{\lm(P), \mu}.
\end{equation}
\end{lemma}

\begin{proof}
Recall that for an integer partition $\lm = (\lm_1,\dots,\lm_k)$, we define $G_{\lm}$ to have stable sets of size $\lm_1,\dots,\lm_k$, and all possible edges between vertices in distinct stable sets. Label the vertices of $G_{\lm(P)}$ with elements of $S$ such that the blocks of the stable partition in $G_{\lm(P)}$ of type $\lm(P)$ correspond to the blocks of $P$. Note that for any $\pi \vdash S$, we have $\pi \leq P$ if and only if the corresponding $\pi \vdash V(G_{\lm(P)})$ is stable. Since by construction $X_{G_{\lm(P)}} = r_{\lm(P)}$, we have
$$
\sum_{\pi \leq P} [r_{\mu}]\mt_{\lm(\pi)} = \sum_{\substack{\pi \vdash G_{\lm(P)} \\ \pi \text{ stable}}} [r_{\mu}]\mt_{\lm(\pi)} = [r_{\mu}]X_{G_{\lm(P)}} = \delta_{\lm(P), \mu}.
$$
\end{proof}

We now prove the main theorem of this section.

\begin{theorem}\label{thm:sweet}
Let $j$ be a nonnegative integer, and suppose that $\{M_1,\dots,M_k\}$ is the set of all partitions of $V(G)$ such that each $M_t$ is $i$-maximal for some $0 \leq i \leq j$.  Then
$$
\sum_{i=0}^j [(1+t)^i]XB_G = \sum_{\textnormal{nonempty}\, S \subseteq [k]} (-1)^{|S|-1}r_{\lm(\wedge_{i \in S} M_i)}.
$$
\end{theorem}

\begin{proof}

It suffices to show that for any $\mu$ 
$$
\sum_{i=0}^j [(1+t)^ir_{\mu}]XB_G = \sum_{\substack{S \subseteq [k] \\ \lm(\wedge_{i \in S} M_i) = \mu}} (-1)^{|S|-1}.
$$

Starting from the $\mt$-basis expansion \eqref{eq:xbtom} we have
$$
XB_G = \sum_{\pi \vdash V(G)} (1+t)^{e(\pi)} \mt_{\lm(\pi)} = \sum_{\pi \vdash V(G)} (1+t)^{e(\pi)} \sum_{\mu} \left([r_{\mu}]\mt_{\lm(\pi)}\right)r_{\mu}.
$$
Thus
\begin{equation}\label{eq:1plust}
\sum_{i=0}^j [(1+t)^ir_{\mu}]XB_G = \sum_{\substack{\pi \vdash V(G) \\ e(\pi) \leq j}} [r_{\mu}]\mt_{\lm(\pi)}.
\end{equation}
Every such $\pi$ must satisfy $\pi \leq M_i$ for all $i$ in some nonempty subset of $[k]$. Thus, we may apply inclusion-exclusion to these subsets and simplify the right-hand side of \eqref{eq:1plust} to
\begin{equation}\label{eq:fin}
\sum_{l=1}^k (-1)^{l-1} \sum_{\substack{ S \subseteq [k] \\ |S| = l}} \, \sum_{\pi \leq \wedge_{i \in S} M_i} [r_{\mu}]\mt_{\lm(\pi)} = \sum_{\textnormal{nonempty} \, S \subseteq [k]} (-1)^{|S|-1}\sum_{\pi \leq \wedge_{i \in S} M_i} [r_{\mu}]\mt_{\lm(\pi)}
\end{equation}
Applying Lemma \ref{lem:lemcool} with $P = \wedge_{i \in S} M_i$ for each nonempty $S \subseteq [k]$ on the right-hand side of \eqref{eq:fin}, we get
$$
\sum_{\textnormal{nonempty} \, S \subseteq [k]} (-1)^{|S|-1}\delta_{\lm(\wedge_{i \in S} M_i),\mu} = 
 \sum_{\substack{S \subseteq [k] \\ \lm(\wedge_{i \in S} M_i) = \mu}} (-1)^{|S|-1}.
$$
as desired.
\end{proof}
We call a partition of $G$ a  \emph{maximal stable partition} if it is $0$-maximal.
\begin{cor}
Let $G$ be a graph, and let $M_1,\dots,M_k$ be the maximal stable partitions of $G$.  Then  
$$
X_G = \sum_{\textnormal{nonempty}\, S \subseteq [k]} (-1)^{|S|-1}r_{\lm(\wedge_{i \in S} M_i)}.
$$
\end{cor}
\begin{proof}
Since $XB_G$ reduces to $X_G$ when $t = -1$, we have that $X_G = [(1+t)^0]XB_G$, so the proof follows from letting $j = 0$ in Theorem \ref{thm:sweet}.
\end{proof}

\section{Relationships between the $r$-basis and other symmetric function bases}

We first consider the relationship of the $r$-basis to the $\mt$-basis. To do so, we will often work in a slightly modified version of the ring of symmetric functions denoted by $\widetilde{\Lambda}$ in which we retain the same addition operation but use a different multiplication operation $\otimes$ defined by $\mt_{\lm} \otimes \mt_{\mu} = \mt_{\lm \sqcup \mu}$, where $\lm \sqcup \mu$ is the partition whose parts are the disjoint union of the parts of $\lm$ and $\mu$, e.g. $(3,1,1) \sqcup (2,1) = (3,2,1,1,1)$.  This ring and the associated multiplication were introduced and studied by Tsujie \cite{tsu}.

Tsujie noted that the action of the operation $\otimes$ on the chromatic symmetric function has an interpretation in terms of the \emph{join} of two graphs $G$ and $H$, defined as the graph $G \otimes H$ with vertex set $V(G) \sqcup V(H)$ and edge set $E(G) \sqcup E(H) \sqcup \{vw : v \in G, w \in H\}$ (where here all $\sqcup$s mean disjoint union).
\begin{lemma}\label{lem:ot}(\cite{tsu})

  Let $G,H$ be graphs.  Then
  $$
  X_G \otimes X_H = X_{G \otimes H}.
  $$
  
\end{lemma}

\begin{proof}

  It suffices to show that both sides have the same coefficient of $\mt_{\lm}$ for all $\lm \vdash |V(G)|+|V(H)|$.  For a fixed $\lm$, every $\pi \in Stab_{\lm}(G \otimes H)$ must be of the form $A \sqcup B$, where $A$ is a partition of $V(G)$ whose parts are a submultiset of the multiset of the parts of $\lm$, and $B$ is a partition of $V(H)$ whose parts are the remainder of $\lm$.  Thus, $|Stab_{\lm}(G \otimes H)|$ (and by \eqref{eq:xgmt} the desired coefficient) is
  $$
  [\mt_{\lm}]X_{G \otimes H} = \sum_{\mu \subseteq \lm} [\mt_{\mu}]X_G[\mt_{\lm \backslash \mu}]X_H
  $$
  where $\mu \subseteq \lm$ means that the parts of $\mu$ are a submultiset of the parts of $\lm$, and $\lm \bk \mu$ denotes the partition comprising the parts of $\lm$ with the parts of $\mu$ removed.  By the definition of $\otimes$ this is also equal to
  $$
  [\mt_{\lm}](X_G \otimes X_H)
  $$
  so we are done.

\end{proof}

It follows from Lemma \ref{lem:ot} and the definition of the $r_{\lm}$ that\footnote{Note that the $r_{\lm}$ are not multiplicative with respect to the usual multiplication of $\Lambda$.}
$$
r_{\lm} \otimes r_{\mu} = r_{\lm \sqcup \mu}.
$$

For a partition $\lm = (\lm_1,\dots,\lm_k)$ of $n$, let $SP(n,\lm)$ denote the number of partitions of $[n]$ into parts of sizes $\lm_1,\dots,\lm_k$.  Since $r_n = p_{1^n}$ is the chromatic symmetric function of the graph of $n$ vertices with no edges, from \eqref{eq:xgmt} we get
\begin{equation}\label{eq:rtom}
r_n = \sum_{\mu \vdash n} |SP(n,\mu)|\mt_{\mu} =  \sum_{\mu \vdash n} \frac{n!}{\prod_i \mu_i! \prod_i n_i(\mu)!}\mt_{\mu}
\end{equation}
and therefore, in the modified ring $\widetilde{\Lambda}$, we have
$$
r_{\lm} = r_{\lm_1} \otimes r_{\lm_2} \otimes \dots \otimes r_{\lm_k} =
$$
\begin{equation}\label{eq:btom}
\left(\sum_{\mu^1 \vdash \lm_1} \frac{\lm_1!}{\prod_i \mu^1_i! \prod_i n_i(\mu^1)!}\mt_{\mu^1}\right) \otimes \dots \otimes \left(\sum_{\mu^k \vdash \lm_k} \frac{\lm_k!}{\prod_i \mu^k_i! \prod_i n_i(\mu^k)!}\mt_{\mu^k}\right).
\end{equation}

For partitions $\lm = (\lm_1,\dots,\lm_k)$ and $\mu = (\mu_1,\dots,\mu_m)$, define a \emph{puzzle of $\mu$ into $\lm$} to be an ordered tuple of partitions $(\mu^1,\dots,\mu^k)$ such that
\begin{itemize}
\item For all $1 \leq i \leq k$ we have $\mu^i \vdash \lm_i$.
\item The disjoint union of the parts of the $\mu^i$ is $\mu$.
\end{itemize}
If there exists at least one puzzle of $\mu$ into $\lm$, we say that $\mu$ \emph{refines} $\lm$ or is a \emph{refinement} of $\lm$, and that $\lm$ \emph{coarsens} $\mu$ or is a \emph{coarsening} of $\mu$, and we write $\mu \leq \lm$.

In terms of puzzles, we may use \eqref{eq:btom} to extract the coefficient of $\mt_{\mu}$ in the expansion of $r_{\lm}$ as
\begin{equation}\label{eq:bm}
[\mt_{\mu}]r_{\lm} = \frac{\prod_i \lm_i!}{\prod_i \mu_i!} \sum_{\textnormal{puzzles}\, \mu \rightarrow \lm} \frac{1}{\prod_{i,j} n_i(\mu^j)!}.
\end{equation}

Note that this value is  $0$ if the sum is empty, so $[\mt_{\mu}]r_{\lm} = 0$ when $\mu$ is not a refinement of $\lm$.  Therefore the linear transformation mapping the $\mt$-basis to the $r$-basis is (with respect to the reverse lexicographic order on partitions) an upper triangular matrix with $1$s on the main diagonal, so it is invertible.  Thus, for each $d$, the set $\{r_{\lm} : \lm \vdash d\}$ is a basis for $\widetilde{\Lambda}^d$ and so also for $\Lambda^d$, as we noted in Section 3 (and as was originally proved by Penaguiao \cite{raul}).  Furthermore, passing back to the usual $m$-basis, we have
\begin{equation}\label{eq:bmono}
[m_{\mu}]r_{\lm} = \frac{\prod_i \lm_i!}{\prod_i \mu_i!} \sum_{\textnormal{puzzles}\, \mu \rightarrow \lm} \frac{\prod_i n_i(\mu)!}{\prod_{i,j} n_i(\mu^j)!} = \frac{\prod_i \lm_i!}{\prod_i \mu_i!}R_{\lm\mu}
\end{equation}
where $R$ is the change-of-basis matrix from the $m$-basis to the $p$-basis in $\Lambda^d$ (so the rows and columns are indexed by partitions of $d$), with entries given by $[m_{\mu}]p_{\lm} = R_{\mu\lm}$ \cite{mac, stanleybook}.  Thus the $r$-basis fits naturally into the framework of fundamental symmetric function bases.  

Furthermore, we may use \eqref{eq:bmono} to provide a combinatorial interpretation to the entries of the inverse matrix going from the $\mt$-basis to the $r$-basis. We provide two proofs of this formula, one algebraically (due to Panova \cite{panova}) and the other enumeratively, with each proof highlighting different aspects of the $r$-basis. For the algebraic proof, we first need a simple expression for the change-of-basis coefficients from the $p$-basis to the $h$-basis.  We may verify (e.g. as in \cite{stanleybook}) that
$$
[h_{\mu}]p_n = n\sum_{\mu \vdash n} \frac{(-1)^{l(\mu)-1}(l(\mu)-1)!}{\prod_i n_i(\mu)!}.
$$

Extending this to general $p_{\lm}$ by multiplication gives
  \begin{equation}\label{eq:ph}
[h_{\mu}]p_{\lm} = \left(\prod_j \lm_j\right) \sum_{\textnormal{puzzles}\, \mu \rightarrow \lm} \prod_i \frac{(-1)^{l(\mu^i)-1}(l(\mu^i)-1)!}{\prod_j n_j(\mu^i)!}.
  \end{equation}
  
\begin{theorem}\label{thm:necklace}

\begin{equation}\label{eq:necklace}
m_n = \sum_{\mu \vdash n} (-1)^{l(\mu)-1} c_{\mu}r_{\mu}
\end{equation}
where
$$
c_{\mu} = \frac{n!(l(\mu)-1)!}{\prod_i \mu_i!\prod_i n_i(\mu)!}
$$
is the number of cyclically ordered set partititons of $[n]$ of type $\mu$.  That is, $c_{\mu}$ is the number of ways to make a circular necklace with distinguishable beads \\ $B_1,\dots,B_{l(\mu)}$ such that
\begin{itemize}
\item Each bead $B_i$ contains a nonempty subset $S(B_i)$ of $\{1,2,\dots,n\}$.
  \item The $S(B_i)$ form a partition of $\{1,2,\dots,n\}$.
  \item The multiset $\{|S(B_1)|,\dots,|S(B_{l(\mu)})|\}$ is exactly the multiset of parts of $\mu$.
\end{itemize}

\end{theorem}

\begin{proof}[First Proof (\cite{panova})]

  By inverting the matrix equation \eqref{eq:bmono} we have that
  $$
[r_{\lm}]m_{\mu} = \frac{\prod_i \mu_i!}{\prod_i \lm_i!}\left(R^{-1}\right)_{\mu\lm}.
$$
Comparing to the desired equality \eqref{eq:necklace}, if we set $\mu = (n)$, and relabel $\lm$ as $\mu$ for consistency with other relations, it suffices to show that
$$
\frac{n!}{\prod_i \mu_i!}\left(R^{-1}\right)_{(n)\mu} =  n! \frac{(-1)^{l(\mu)-1}(l(\mu)-1)!}{\prod_i \mu_i!\prod_i n_i(\mu)!}
$$
or after simplifying, that
\begin{equation}\label{eq:rnm}
\left(R^{-1}\right)_{(n)\mu} = \frac{(-1)^{l(\mu)-1}(l(\mu)-1)!}{\prod_i n_i(\mu)!}.
\end{equation}

We will show \eqref{eq:rnm} directly using the definition of $R$ as the change-of-basis coefficients from $p$ to $m$ and the Cauchy identity \cite{mac, stanleybook}
\begin{equation}\label{eq:cauchy}
\prod_{i,j \geq 1} \frac{1}{1-x_iy_j} = \sum_{\mu} m_{\mu}(x)h_{\mu}(y) = \sum_{\lm} \frac{p_{\lm}(x)p_{\lm}(y)}{\prod_i \lm_i \prod_i n_i(\lm)!}.
\end{equation}
where $x$ and $y$ are to be interpreted as shorthand for the countably many variables $x_1,x_2,\dots$ and $y_1,y_2,\dots$ respectively.
First, we expand the $p_{\lm}(y)$ in the right-hand sum of \eqref{eq:cauchy} into the $h$-basis using \eqref{eq:ph}.  Then we consider the middle and rightmost sums of \eqref{eq:cauchy} as symmetric functions in the $y$ variables with coefficients in $\mathbb{R}[[x_1,x_2,\dots]]$, expand them in the $h$-basis, and equate the coefficients of $h_{\mu}(y)$ on both sides. We obtain
$$
m_{\mu}(x) = \sum_{\lm} \frac{p_{\lm}(x)}{\prod_i \lm_i \prod_i n_i(\lm)!}\left(\prod_j \lm_j\right) \sum_{\textnormal{puzzles}\, \mu \rightarrow \lm} \prod_{i} \frac{(-1)^{l(\mu^i)-1}(l(\mu^i)-1)!}{\prod_j n_j(\mu^i)!}.
$$

Since $[p_{\lm}]m_{\mu} = (R^{-1})_{\lm\mu}$, we have
$$
\left(R^{-1}\right)_{\lm\mu} = \frac{1}{\prod_i \lm_i \prod_i n_i(\lm)!}\left(\prod_j \lm_j\right) \sum_{\textnormal{puzzles}\, \mu \rightarrow \lm} \prod_{i} \frac{(-1)^{l(\mu^i)-1}(l(\mu^i)-1)!}{\prod_j n_j(\mu^i)!}.
$$
Setting $\lm = (n)$ gives exactly \eqref{eq:rnm}, so we are done.

\end{proof}

\begin{proof}[Second Proof]
We proceed by induction. The base case of $\mt_1 = r_1$ is easily checked, so we now wish to show the claim for $n \geq 2$ under the assumption that the claim is known for all positive integers less than $n$. Starting from equation \eqref{eq:rtom} for $[\mt_{\lm}]r_n$, we may rearrange terms and apply the inductive hypothesis to get
\begin{align*}
\widetilde{m_n} &= r_n - \sum_{\lm \neq (n)} \frac{n!}{\prod_i \lm_i! \prod_i n_i(\lm)!}\mt_{\lm} \\ &= 
r_n -  \sum_{\lm \neq (n)} \frac{n!}{\prod_i \lm_i! \prod_i n_i(\lm)!} \sum_{\textnormal{puzzles}\, \mu \rightarrow \lm} \bigotimes_i  (-1)^{l(\mu^i)-1}\frac{\lm_i!(l(\mu^i)-1)!)}{\prod_j \mu_{j}^i! \prod_j n_j(\mu^i)!}r_{\mu^i}.
\end{align*}

Thus $[r_n]\mt_n = 1$, and for $\mu \neq (n)$
$$
[r_{\mu}]\mt_n = -\sum_{\lm \neq (n)} \frac{n!}{\prod_i \lm_i! \prod_i n_i(\lm)!} \sum_{\textnormal{puzzles}\, \mu \rightarrow \lm} (-1)^{l(\mu)-l(\lm)} \prod_i \frac{\lm_i!(l(\mu^i) - 1)!}{\prod_j \mu_j^i!\prod_j n_j(\mu^i)!}
$$

Comparing to the desired formula \eqref{eq:necklace}, it suffices to show that for $\mu \neq (n)$ we have
$$
 (-1)^{l(\mu)-1}\frac{n!(l(\mu)-1)!}{\prod_i \mu_i!\prod_i n_i(\mu)!}  =  $$ $$-\sum_{\lm \neq (n)} \frac{n!}{\prod_i \lm_i! \prod_i n_i(\lm)!} \sum_{\textnormal{puzzles} \, \mu \rightarrow \lm} (-1)^{l(\mu)-l(\lm)}\prod_i \frac{\lm_i! (l(\mu^i) - 1)!}{\prod_j \mu_j^i!\prod_j n_j(\mu^i)!}.
$$

After cancelling common terms from both sides, as well as some common terms from the numerator and denominator of the right-hand side, we reduce to
\begin{equation}\label{eq:flip}
\frac{(l(\mu)-1)!}{\prod_i n_i(\mu)!} = \sum_{\lm \neq (n)} \frac{(-1)^{l(\lm)}}{\prod_i n_i(\lm)!} \sum_{\textnormal{puzzles}\, \mu \rightarrow \lm} \prod_i \frac{(l(\mu^i)-1)!}{\prod_j n_j(\mu^i)!}.
\end{equation}

We note that the left-hand side of \eqref{eq:flip} is just $-1$ times the missing $\lm = (n)$ case of the sum on the right-hand side, so by subtracting it from both sides we simplify the equality to show for $\mu \neq (n)$ to
\begin{equation}\label{eq:small}
 \sum_{\lm \vdash n} \frac{(-1)^{l(\lm)}}{\prod_i n_i(\lm)!} \sum_{\textnormal{puzzles}\, \mu \rightarrow \lm} \prod_i \frac{ (l(\mu^i)-1)!}{\prod_j n_j(\mu^i)!} = 0.
\end{equation}
 
 We prove this by considering a more general formula.  For every partition $\mu \vdash n$ and every function $f: \mathbb{N} \rightarrow \mathbb{R}$ we will find a simple evaluation of
$$
\sum_{\lm \vdash n} \frac{f(l(\lm))}{\prod_i n_i(\lm)!} \sum_{\textnormal{puzzles}\, \mu \rightarrow \lm} \prod_i \frac{ (l(\mu^i)-1)!}{\prod_j n_j(\mu^i)!}.
$$

First we modify the equation by multiplying by $\frac{n!\prod \lm_i}{n!\prod \lm_i}$, to get
\begin{equation}\label{eq:initial}
\frac{1}{n!}\sum_{\lm \vdash n} f(l(\lm))\frac{n!}{\prod_i n_i(\lm)! \prod_i \lm_i} \sum_{\textnormal{puzzles}\, \mu \rightarrow \lm} \prod_i \frac{(l(\mu^i)-1)!\lm_i}{\prod_j n_j(\mu^i)!}.
\end{equation}

We note here an important lemma that will be used repeatedly throughout this proof: let $S_n^{\lm}$ denote the set of permutations of $[n]$ that have cycle type $\lm$ (meaning that when written as a product of disjoint cycles, the partition of $[n]$ induced by dropping the ordering on the cycles is an element of $SP(n, \lm)$). We may form such a permutation by taking any of the $n!$ permutations of $[n]$, writing it as a list, and taking the first $\lm_1$ elements in order as a cycle, the next $\lm_2$ elements in order as a cycle, and so on.  Then we will obtain each element of $S_n^{\lm}$  with multiplicity equal to $\prod_i n_i(\lm)!\prod_i \lm_i$, because cycles of the same size may be interchanged, and within each cycle any cyclic permutation of the given ordering leads to the same cycle as a permutation.  Thus
\begin{equation}\label{eq:cyctype}
|S_n^{\lm}| = \frac{n!}{\prod_i n_i(\lm)!\prod_i \lm_i}.
\end{equation}

Therefore, the outer sum of \eqref{eq:initial} may be viewed as taking for each $\lm \vdash n$ all permutations with cycle type $\lm$, weighted by a factor of $f(l(\lm))$.  So we may rewrite this equation as
\begin{equation}\label{eq:step2}
\frac{1}{n!}\sum_{\lm \vdash n} f(l(\lm)) \sum_{\pi \in S_n^{\lm}} \sum_{\textnormal{puzzles}\, \mu \rightarrow \lm} \prod_i \frac{(l(\mu^i)-1)!\lm_i}{\prod_j n_j(\mu^i)!}.
\end{equation}

To further interpret \eqref{eq:step2} we will consider the combinatorics of necklaces and introduce some terminology. Every permutation $\pi$ has a disjoint cycle decomposition that may be visualized as a set of necklaces, each with beads corresponding to elements of $\{1,2,\dots,n\}$, and with a prescribed orientation for all (e.g. clockwise). For example, we would associate the permutation $\pi_0 = (1348)(257)(6) \in S_8^{(4,3,1)}$ with the set of three necklaces, one with four beads labelled $1 \rightarrow 3 \rightarrow 4 \rightarrow 8 \rightarrow 1$ going clockwise, one with three beads labelled $2 \rightarrow 5 \rightarrow 7 \rightarrow 2$ going clockwise, and one with one bead labelled $6$. Define a \emph{chop} of a necklace to be a choice of a space between consecutive beads, and a \emph{cut} of a necklace to be a set of chops (thus, a necklace with $m$ beads has $m$ possible chops and $2^m$ possible cuts). A cut of a necklace with $m \geq 1$ chops separates it into $m$ \emph{strings}, or ordered sequences of beads, where each string is formed by starting at a chop and continuing clockwise until the next chop. We define the \emph{reassembly} of a cut necklace to be the set of necklaces formed by joining the endpoints of each string such that the beads viewed clockwise are in the same order as they were in the string. Thus, given a permutation's necklace representation, we may cut and reassemble the necklaces and get the necklace representation of a new permutation. If $\pi, \sigma \in S_n$, we say that $\pi$ \emph{may be cut into} $\sigma$, notated $\pi \succ \sigma$, if there is a choice of cuts of the cycle necklaces of $\pi$ such that their reassembly forms $\sigma$. Let $ct(\pi)$ denote the \emph{cycle type} of $\pi$, meaning the integer partition given by the lengths of the cycles of $\pi$.  For $\pi$ to have a cut into $\sigma$, it is clearly necessary (but not sufficient) that $ct(\sigma) \leq ct(\pi)$. Let $c$ be a cycle occurring as a member of the disjoint cycle decomposition of $\pi$. We say that $c$ is a \emph{shared cycle} of $\pi$ and $\sigma$ if $c$ also occurs as a member of the disjoint cycle decomposition of $\sigma$.  We also let $|c|$ denote the number of elements in $c$, or the number of beads in the necklace representing $c$ in a permutation.

Now, we claim that \eqref{eq:step2} is equal to
\begin{equation}\label{eq:step}
\frac{1}{n!}\sum_{\lm \vdash n} f(l(\lm)) \sum_{\pi \in S_n^{\lm}} \sum_{\substack{\sigma \prec \pi \\ \sigma \in S_n^{\mu}}}  \prod_{\substack{c \text{ a shared} \\ \text{cycle of } \pi, \sigma}} |c|.
\end{equation}
That is, we are rewriting the inner sum of \eqref{eq:step2} (as a function of $\mu$, and of the middle sum's choice of $\pi$). To justify this, we start from \eqref{eq:step} and derive \eqref{eq:step2}.  First, to cut a $\sigma \in S_n^{\mu}$ from $\pi$, fix an ordering of the cycles of $\pi$ with lengths non-increasing, and choose a puzzle of $\mu$ into $\lm$. For each cycle of $\pi$, this puzzle dictates the sizes of the strings to make when cutting the cycle necklaces.

Now we dictate exactly how the chops are made.  First, for each cycle of $\pi$, we choose an initial chop. This may be done in $\prod_i \lm_i$ ways. Now, for each $i$ let $\mu^i$ be the partition corresponding to $\lm_i$ in the chosen puzzle. We specify the chops continuing clockwise from the initial chop of cycle $i$ of $\pi$ by choosing an ordering of the parts of $\mu^i$ in one of $\frac{l(\mu^i)!}{\prod_j n_j(\mu^i)!}$ ways. We do this for each cycle separately.

Note now how many times we have counted each permutation $\sigma$. Each $\sigma$ is uniquely determined by the set of cuts of the cycles of $\pi$, excepting those cycles receiving only the initial chop. However, $\sigma$ is independent of the order of the chops, so for each $1 \leq i \leq l(\lm)$ we have overcounted by a factor of $l(\mu^i)$ since for cycle $i$ of $\pi$ either it only received one chop, in which case $l(\mu^i) = 1$, or we may make the same cut with a different initial chop for that cycle, and a corresponding clockwise ordering of the parts of $\mu^i$ from that initial chop. Hence we may divide by this factor. Additionally, each cycle $c$ of $\pi$ that is unchanged in $\sigma$ uses only the initial chop, and it does not matter which chop is made, leading to overcounting by a factor of $|c|$. Taking all of these factors into account gives the desired equality of \eqref{eq:step2} and \eqref{eq:step}.

Now, we can rewrite \eqref{eq:step} by rearranging the sum.  Instead of choosing $\pi \in S_n^{\lm}$ and then $\sigma \prec \pi$ with $\sigma \in S_n^{\mu}$, we may first choose $\sigma$ and then a $\pi$ that cuts into $\sigma$. We write this sum now as
\begin{equation}\label{eq:rearrange}
   \frac{1}{n!}\sum_{\sigma \in S_n^{\mu}} \sum_{\lm \vdash n} f(l(\lm)) \sum_{\substack{\pi \succ \sigma \\ \pi \in S_n^{\lm}}}  \prod_{\substack{c \text{ a shared} \\ \text{cycle of } \pi, \sigma}} |c|.
\end{equation}

Now, every $\pi \succ \sigma$, regardless of cycle type, may be obtained by reversing the cut and reassembly process.  We may specify a chop of each cycle of $\sigma$, and a permutation $\tau$ of these cycles (thus we view $\tau$ as lying in $S_{l(\mu)}$). For each choice of $\sigma$, $\tau$, and one of the $\prod_{\text{cycle c of } \sigma} |c|$ ways to chop each cycle of $\sigma$ into a string, we may form a permutation $\pi$ by gluing the strings of $\mu$ together as prescribed by $\tau$.  For example, if $\mu = (1346)(25)(79)(8)$, we may select chops between $3$ and $4$, between $5$ and $2$ (in that order), between $7$ and $9$ (in that order), and the trivial chop on the other necklace, forming strings $\overrightarrow{4613}$, $\overrightarrow{52}$, $\overrightarrow{79}$, and $\overrightarrow{8}$. Then we may select $\tau = (142)(3) \in S_4$, which means that we attach the end of the first string to the beginning of the fourth, the end of the fourth string to the beginning of the second, and the end of the second to the beginning of the first, obtaining the necklace (and cycle) $(4613852)$.  Then the remaining string is just reattached into the necklace $(79)$ it was previously.

In doing this, note that for the resulting $\pi$, $f(l(ct(\pi))) = f(l(ct(\tau))$, so that we may change to summing over $\tau$. Also, note that each $\pi$ is overcounted by a factor of the product of the lengths of the cycles unchanged from $\sigma$, since the chops in those cycles did not affect the reassembly.  Thus, we have shown that \eqref{eq:rearrange} is equal to
$$
\frac{1}{n!}\sum_{\sigma \in S_n^{\mu}} \sum_{\tau \in S_{l(\mu)}} f(l(ct(\tau))) \prod_{\text{ cycles c of } \sigma} |c| = \frac{1}{\prod_i n_i(\mu)!} \sum_{\tau \in S_{l(\mu)}} f(l(ct(\tau))).
$$
Thus, overall, we have shown that for every $\mu \vdash n$
\begin{equation}\label{eq:stepnice}
\sum_{\lm \vdash n} \frac{f(l(\lm))}{\prod_i n_i(\lm)!} \sum_{\textnormal{puzzles}\, \mu \rightarrow \lm} \prod_i \frac{ (l(\mu^i)-1)!}{\prod_j n_j(\mu^i)!} = \frac{1}{\prod_i n_i(\mu)!} \sum_{\tau \in S_{l(\mu)}} f(l(ct(\tau))).
\end{equation}

Returning to \eqref{eq:small}, the equality we must show to prove the theorem, we apply \eqref{eq:stepnice} with $f(l(\lm))= (-1)^{l(\lm)}$. For each $\mu \neq (n)$, we have $\sum_{\tau \in S_{l(\mu)}} (-1)^{l(ct(\tau))} = 0$, since there are equally many odd and even permutations of $S_k$ for $k \geq 2$, and thus the entire expression evaluates to $0$, completing the proof.

\end{proof}

Furthermore, the equations derived throughout this second proof also provide combinatorial identities that are interesting in their own right. For example, in \eqref{eq:stepnice}, letting $f(l(\lm)) = 1$ for all $\lm$ yields
$$\sum_{\lm \vdash n} \frac{1}{\prod_i n_i(\lm)!} \sum_{\textnormal{puzzles}\, \mu \rightarrow \lm} \prod_i \frac{ (l(\mu^i)-1)!}{\prod_j n_j(\mu^i)!} = \frac{1}{\prod_i n_i(\mu)!} \sum_{\tau \in S_{l(\mu)}} 1 = 
\frac{l(\mu)!}{\prod_i n_i(\mu)!}
$$
where this last value is the number of integer compositions (partitions with any ordering of the parts) whose parts are those of $\mu$.

The change-of-basis matrix from the $r$-basis to any other fundamental basis may be derived from its transition matrices to and from the $m$-basis, given that the matrices for transitioning between the $m$-basis and the other fundamental bases are well-known \cite{mac, stanleybook}. In the remainder of this section, we highlight a reciprocity relation between expanding the $\mt$-basis in terms of the $r$-basis, and expanding the $p$-basis in terms of the $e$-basis.  This may be viewed in relation to the chromatic symmetric function, since in this setting the $\mt$- and $p$-bases are graph complements of each other (in the context of vertex-weighted graphs, see \cite{delcon}), as are the $e$- and $r$-bases by definition.

\begin{theorem}\label{thm:recip}
For integer partitions $\lm$ and $\mu$
\begin{equation}\label{eq:eprec}
(-1)^{|\lm|-l(\lm)}\frac{\prod_i (\lm_i-1)!}{\prod_i \mu_i!}[e_{\mu}]p_{\lm} = [r_{\mu}]\mt_{\lm}
\end{equation}
and
\begin{equation}\label{eq:perec}
(-1)^{|\lm|-l(\lm)}\frac{\prod_i \lm_i!}{\prod_i (\mu_i-1)!}[p_{\mu}]e_{\lm} = [\mt_{\mu}]r_{\lm}.
\end{equation}
\end{theorem}

\begin{proof}
As a result of Theorem \ref{thm:necklace}, we have
$$
\mt_n = \sum_{\mu \vdash n} (-1)^{l(\mu)-1}\frac{n!(l(\mu)-1)!}{\prod_i \mu_i! \prod_i n_i(\mu)!}r_{\mu}
$$
and so proceeding by multiplication using $\otimes$ we have
\begin{align*}
[r_{\mu}]\mt_{\lm} &= \sum_{\textnormal{puzzles}\, \mu \rightarrow \lm} \prod_i (-1)^{l(\mu^i)-1}\frac{\lm_i!(l(\mu^i)-1)!}{\prod_j \mu_j^i! \prod_j n_j(\mu^i)!} \\ &=
(-1)^{l(\mu)-l(\lm)}\frac{\prod_i \lm_i!}{\prod_i \mu_i!}  \sum_{\textnormal{puzzles}\, \mu \rightarrow \lm} \prod_i \frac{(l(\mu^i)-1)!}{\prod_j n_j(\mu^i)!}
\end{align*}

On the other hand, it is known \cite{mac, stanleybook} that
$$
p_n = \sum_{\mu \vdash n} (-1)^{n-l(\mu)} \frac{n(l(\mu)-1)!}{\prod_i n_i(\mu)!}e_{\mu}
$$
so it follows through the usual multiplication that
\begin{align*}
[e_{\mu}]p_{\lm} &= \sum_{\textnormal{puzzles}\, \mu \rightarrow \lm} \prod_i (-1)^{\lm_i-l(\mu^i)} \frac{\lm_i (l(\mu^i)-1)!}{\prod_j n_j(\mu^i)!} \\  &=
 (-1)^{|\lm|-l(\mu)} \left(\prod_i \lm_i\right) \sum_{\textnormal{puzzles}\, \mu \rightarrow \lm} \prod_i \frac{(l(\mu^i)-1)!}{\prod_j n_j(\mu^i)!}.
\end{align*}
Comparing these evaluations of $[r_{\mu}]\mt_{\lm}$ and $[e_{\mu}]p_{\lm}$ yields \eqref{eq:eprec}.

Likewise, from \eqref{eq:rtom} we have
$$
r_n = \sum_{\mu \vdash n}  \frac{n!}{\prod_i \mu_i! \prod_i n_i(\mu)!} \mt_{\mu}
$$
so that
\begin{align*}
[\mt_{\mu}]r_{\lm} &= \sum_{\textnormal{puzzles}\, \mu \rightarrow \lm} \prod_i \frac{\lm_i!}{\prod_j \mu_j^i! \prod_j n_j(\mu^i)!} \\ &= 
\frac{\prod_i \lm_i!}{\prod_i \mu_i!} \sum_{\textnormal{puzzles}\, \mu \rightarrow \lm} \prod_i \frac{1}{\prod_j n_j(\mu^i)!}.
\end{align*}

Similarly, it is known \cite{mac, stanleybook} that
$$
e_n = \sum_{\mu \vdash n} \frac{(-1)^{n-l(\mu)}}{\prod_i \mu_i \prod_i n_i(\mu)!}p_{\mu}
$$
So we derive that
\begin{align*}
[p_{\mu}]e_{\lm} &= \sum_{\textnormal{puzzles}\, \mu \rightarrow \lm} \prod_i \frac{(-1)^{\lm_i-l(\mu^i)}}{\prod_j \mu_j^i \prod_j n_j(\mu^i)!} \\ &= 
(-1)^{|\lm|-l(\mu)}\frac{1}{\prod_i \mu_i} \sum_{\textnormal{puzzles}\, \mu \rightarrow \lm} \prod_i \frac{1}{\prod_j n_j(\mu^i)!}.
\end{align*}

Comparing these evaluations of $[\mt_{\mu}]r_{\lm}$ and $[p_{\mu}]e_{\lm}$ yields \eqref{eq:perec}.

\end{proof}

\end{section}

\section{Further Research}

The results presented here demonstrate both how natural and how important the $r$-basis is in the context of symmetric functions. There is still much to be explored in regards to the relationships between the $r$-basis and other fundamental bases, and whether these relationships might give rise to further combinatorial descriptions, such as in Theorem \ref{thm:necklace}. It is also worth considering the $r$-basis in relation to the linear forest and star forest bases considered in \cite{ali}, which are closely related to ribbon Schur functions.

The $r$-basis seems particularly powerful as a tool to consider the chromatic and Tutte symmetric functions, as evidenced by Theorem \ref{thm:sweet}. This theorem can likely be extended to the vertex-weighted versions of these functions introduced in \cite{delcon} and \cite{cstutte} respectively; such a generalized theorem may be related to similar results on the chromatic and Tutte polynomials that may or may not have been discovered yet.

Theorem \ref{thm:sweet} may also be useful in tackling two of the main open problems of the chromatic symmetric function: the Stanley-Stembridge conjecture and the Tree Isomorphism conjecture. With regards to the former, combining Theorem \ref{thm:sweet} with the reciprocity relations \eqref{eq:eprec} and \eqref{eq:perec} of Theorem \ref{thm:recip} gives a new approach for computing the $e$-basis coefficients of $X_G$ and $XB_G$ that may be more approachable than others.  With regards to the latter, Theorem \ref{thm:sweet} gives new information that is encoded by $X_G$ and $XB_G$ about $\Pi(G)$, and especially about maximal stable partitions of $G$; this additional information may be useful for determining whether these functions distinguish nonisomorphic trees.

It would also be interesting to consider the relationship between the $r$-basis and the $U$-polynomial of Noble and Welsh \cite{noble}, a vertex-weighted graph polynomial in variables $x_1, x_2, \dots$, that is equivalent to the Tutte symmetric function. There are already natural relationships between symmetric function bases and the $U$-polynomial (for example, the variable $x_i$ in the $U$-polynomial corresponds to $p_i/t$ in the Tutte symmetric function), and some of the best results towards the Tree Isomorphism Conjecture have been discovered using the $U$-polynomial \cite{trees, trees2}. Thus, finding a natural interpretation of how the $r$-basis manifests in the $U$-polynomial would be interesting.

\begin{section}{Acknowledgments}

The authors would like to thank Greta Panova for helpful discussions, and especially for contributing the algebraic proof of Theorem \ref{thm:necklace} \cite{panova}. The authors would also like to thank Jos\'e Aliste-Prieto and Jos\'e Zamora for helpful discussions and for directing us to \cite{raul}, and Raul Penaguaio and the anonymous referees for helpful suggestions.

This material is based upon work supported by the National Science Foundation under Award No. DMS-1802201.

We acknowledge the support of the Natural Sciences and Engineering Research Council of Canada (NSERC), [funding reference number RGPIN-2020-03912].

Cette recherche a été financée par le Conseil de recherches en sciences naturelles et en génie du Canada (CRSNG), [numéro de référence RGPIN-2020-03912].

\end{section}

\bibliographystyle{plain}
\bibliography{bib}

\begin{thebibliography}{10}

\bibitem{ali}
Farid Aliniaeifard, Victor Wang, and Stephanie van Willigenburg.
\newblock Extended chromatic symmetric functions and equality of ribbon schur
  functions.
\newblock {\em Advances in Applied Mathematics}, 128:102189, 2021.

\bibitem{cstutte}
Jos\'e Aliste-Prieto, Logan Crew, Sophie Spirkl, and Jos\'e Zamora.
\newblock A vertex-weighted {{T}}utte symmetric function, and constructing
  graphs with equal chromatic symmetric function.
\newblock {\em The Electronic Journal of Combinatorics}, pages P2--1, 2021.

\bibitem{trees}
Jos{\'e} Aliste-Prieto, Anna de~Mier, and Jos{\'e} Zamora.
\newblock On trees with the same restricted {{U}}-polynomial and the
  {{P}}rouhet--{{T}}arry--{{E}}scott problem.
\newblock {\em Discrete Mathematics}, 340(6):1435--1441, 2017.

\bibitem{trees2}
Jos{\'e} Aliste-Prieto, Anna de~Mier, and Jos{\'e} Zamora.
\newblock On the smallest trees with the same restricted {{U}}-polynomial and
  the rooted {{U}}-polynomial.
\newblock {\em Discrete Mathematics}, 344(3):112255, 2021.

\bibitem{huh}
Soojin Cho and JiSun Huh.
\newblock On e-positivity and e-unimodality of chromatic quasi-symmetric
  functions.
\newblock {\em SIAM Journal on Discrete Mathematics}, 33(4):2286--2315, 2019.

\bibitem{cho}
Soojin Cho and Stephanie van Willigenburg.
\newblock Chromatic classical symmetric functions.
\newblock {\em Journal of Combinatorics}, 9(2):401--409, 2018.

\bibitem{delcon}
Logan Crew and Sophie Spirkl.
\newblock A deletion--contraction relation for the chromatic symmetric
  function.
\newblock {\em European Journal of Combinatorics}, 89:103143, 2020.

\bibitem{epos}
Samantha Dahlberg, Adrian She, and Stephanie van Willigenburg.
\newblock Schur and $ e $-positivity of trees and cut vertices.
\newblock {\em Electronic Journal of Combinatorics}, 27(1), 2020.

\bibitem{dahl}
Samantha Dahlberg and Stephanie van Willigenburg.
\newblock Lollipop and lariat symmetric functions.
\newblock {\em SIAM Journal on Discrete Mathematics}, 32(2):1029--1039, 2018.

\bibitem{jo}
Joanna~A Ellis-Monaghan and Iain Moffatt~(editors).
\newblock {\em Handbook of the {{T}}utte polynomial}.
\newblock Chapman \& {{H}}all/{{CRC}} {{P}}ress, to appear.

\bibitem{foley}
Ang{\`e}le~M Foley, Ch{\'\i}nh~T Ho{\`a}ng, and Owen~D Merkel.
\newblock Classes of graphs with e-positive chromatic symmetric function.
\newblock {\em Electronic Journal of Combinatorics}, 26.3:P3--51, 2019.

\bibitem{deltazero}
Adriano Garsia, Jim Haglund, Jeffrey~B Remmel, and Meesue Yoo.
\newblock A proof of the {{D}}elta conjecture when $q = 0$.
\newblock {\em Annals of Combinatorics}, 23(2):317--333, 2019.

\bibitem{garsia}
Adriano Garsia and Mark Haiman.
\newblock A graded representation model for {{M}}acdonald's polynomials.
\newblock {\em Proceedings of the National Academy of Sciences},
  90(8):3607--3610, 1993.

\bibitem{delta}
James Haglund, Jeff Remmel, and Andrew Wilson.
\newblock The {{D}}elta conjecture.
\newblock {\em Transactions of the American Mathematical Society},
  370(6):4029--4057, 2018.

\bibitem{factorial}
Mark Haiman.
\newblock Hilbert schemes, polygraphs and the {{M}}acdonald positivity
  conjecture.
\newblock {\em Journal of the American Mathematical Society}, 14(4):941--1006,
  2001.

\bibitem{heil}
Sam Heil and Caleb Ji.
\newblock On an algorithm for comparing the chromatic symmetric functions of
  trees.
\newblock {\em Australasian Journal of Combinatorics}, 75(2):210--222, 2019.

\bibitem{macpaper}
Ian~Grant Macdonald.
\newblock A new class of symmetric functions.
\newblock {\em S{\'e}minaire Lotharingien de Combinatoire [electronic only]},
  20:B20a--41, 1988.

\bibitem{mac}
Ian~Grant Macdonald.
\newblock {\em Symmetric functions and {{H}}all polynomials}.
\newblock Oxford {{U}}niversity {{P}}ress, 1998.

\bibitem{noble}
Steven~D Noble and Dominic~JA Welsh.
\newblock A weighted graph polynomial from chromatic invariants of knots.
\newblock In {\em Annales de l'institut Fourier}, volume~49, pages 1057--1087,
  1999.

\bibitem{ore}
Rosa Orellana and Geoffrey Scott.
\newblock Graphs with equal chromatic symmetric functions.
\newblock {\em Discrete Mathematics}, 320:1--14, 2014.

\bibitem{panova}
Greta Panova.
\newblock Personal communication, 2018.

\bibitem{pau}
Alexander Paunov.
\newblock Planar graphs and {{S}}tanley's chromatic functions.
\newblock {\em arXiv preprint arXiv:1702.05787}, 2017.

\bibitem{paw}
Brendan Pawlowski.
\newblock Chromatic symmetric functions via the group algebra of ${{S}}_n$.
\newblock {\em arXiv preprint arXiv:1802.05470}, 2018.

\bibitem{raul}
Raul Penaguiao.
\newblock The kernel of chromatic quasisymmetric functions on graphs and
  hypergraphic polytopes.
\newblock {\em Journal of Combinatorial Theory, Series A}, 175:105258, 2020.

\bibitem{stanley}
Richard~P Stanley.
\newblock A symmetric function generalization of the chromatic polynomial of a
  graph.
\newblock {\em Advances in Mathematics}, 111(1):166--194, 1995.

\bibitem{stanley2}
Richard~P Stanley.
\newblock Graph colorings and related symmetric functions: ideas and
  applications a description of results, interesting applications, \& notable
  open problems.
\newblock {\em Discrete Mathematics}, 193(1-3):267--286, 1998.

\bibitem{stanleybook}
Richard~P Stanley and S~Fomin.
\newblock Enumerative combinatorics. vol. 2, volume 62 of.
\newblock {\em Cambridge Studies in Advanced Mathematics}, 1999.

\bibitem{tsu}
Shuhei Tsujie.
\newblock The chromatic symmetric functions of trivially perfect graphs and
  cographs.
\newblock {\em Graphs and Combinatorics}, 34(5):1037--1048, 2018.

\bibitem{wang}
David~GL Wang and Monica~MY Wang.
\newblock Non-{{S}}chur-positivity of chromatic symmetric functions.
\newblock {\em arXiv preprint arXiv:2001.00181}, 2020.

\bibitem{zabrocki}
Mike Zabrocki.
\newblock A module for the {{D}}elta conjecture.
\newblock {\em arXiv preprint arXiv:1902.08966}, 2019.

\end{thebibliography}

\end{document}